\documentclass{amsart}
\usepackage{amsfonts}

\newtheorem{theorem}{Theorem}
\theoremstyle{plain}

\newtheorem{corollary}{Corollary}

\newtheorem{definition}{Definition}
\newtheorem{example}{Example}

\newtheorem{proposition}{Proposition}
\newtheorem{remark}{Remark}

\numberwithin{equation}{section}

\begin{document}
\title[Metallic structures on differentiable manifolds]{Metallic structures
on differentiable manifolds}
\author{Mustafa \"{O}zkan}
\address{Gazi University, Faculty of Science, Department of Mathematics\\
06500, Teknikokullar, Ankara, Turkey}
\email{ozkanm@gazi.edu.tr, fatmayilmazsevgi06@gmail.com}
\author{Fatma Y\i lmaz}
\date{}

\begin{abstract}
In this paper, we study metallic structures, i.e. polynomial structures with
the structure polynomial $Q\left( J\right) =J^{2}-aJ-bI$ on manifolds using
the metallic ratio, which is a generalization of the Golden proportion. We
investigate for integrability and parallelism conditions of metallic
structures. Also, we give some properties of the metallic Riemannian metric
and an example of the metallic structure.
\end{abstract}

\maketitle

\section{Introduction}

The Golden ratio is well known by everyone and this ratio possesses rich
application areas such as architecture, painting, design, music, nature,
optimization, and perceptual studies. However, in fact, it is probably fair
to say that the golden ratio has inspired thinkers of all disciplines like
no other number in the history of mathematics.

Crasmareanu and Hretcanu \cite{golden} researched the geometry of the golden
structure on a manifold using a corresponding almost product structure. \"{O}%
zkan \cite{ozkan2014}, \"{O}zkan, \c{C}\i tlak and Taylan \cite{ozkan2015}
and \"{O}zkan and Y\i lmaz \cite{ozkan2016} studied the prolongations of
golden structure. Also, many studies made on golden structures and golden
Riemannian manifolds \cite{gezer, gezer3,savas, sahin}.

In this paper, we use the metallic ratio, which is a generalization of the
golden proportion. This generalization was introduced in 1999 by Vera W. de
Spinadel \cite{spinadel} and named as the metallic means family or metallic
proportions. Recently, many authors studied on metallic structures \cite%
{gezer2,hretcanu2013}.

We are inspired by \cite{golden} and follow its steps. By this way, we use
metallic structure on a differentiable manifold. We are interested in
finding the properties of the metallic structure.

The outline of this paper is as follows. In section 2, we give basic
definitions about metallic means and metallic means family. In the next
section, some relations between the metallic structure and the almost
product structure, almost tangent structure, almost complex structure are
given. Moreover, we show some properties of the metallic structure. In
section 4, we give some examples of the metallic structure. After that, we
study a few connections on the metallic structure. In section 6, we search
for the integrability and parallellism of the metallic structure with
respect to the Schouten and Vr\u{a}nceanu connections. Section 7 deals with
the metallic Riemannian manifold and its properties. At the end of this
section, we give an example of the metallic structure on manifold $%
\mathbb{R}
^{2}$.

\section{Preliminaries}

In this section, we give a brief information about the metallic means family
using the Generalized Secondary Fibonacci sequence.

Fibonacci sequence provides the following relation:%
\begin{equation*}
F(k+1)=F(k)+F(k-1).
\end{equation*}%
This relation can be generalized using 
\begin{equation*}
G(k+1)=aG(k)+bG(k-1),\text{ }k\geq 1
\end{equation*}%
where $a,$ $b,$ $G(0)=c,$ and $G(1)=d$ are real numbers.%
\begin{equation*}
c,\text{ }d,\text{ }ad+bc,\text{ }a(ad+bc)+bd,...
\end{equation*}%
are the members of the Generalized Secondary Fibonacci Sequence.

Let $a$ and $b$ be positive integers. The positive root of the equation%
\begin{equation*}
x^{2}-ax-b=0
\end{equation*}%
is called a member of the metallic means family \cite{spinadel}. This root
is denoted by%
\begin{equation}
\rho _{a,b}=\frac{a+\sqrt{a^{2}+4b}}{2}  \label{2.1}
\end{equation}%
and it is known as metallic ratio.

In (\ref{2.1}), for different values of $a$ and $b$, we obtain the following
ratios:

\begin{itemize}
\item If $a=b=1$, $\rho _{1,1}=\frac{1+\sqrt{5}}{2},$ we get the Golden
mean, which is the ratio of two consecutive classical Fibonacci numbers \cite%
{hretcanu2013},

\item If $a=2$ and $b=1$, $\rho _{2,1}=1+\sqrt{2}$, we get the Silver mean,
which is the ratio of two consecutive Pell numbers \cite{falcon,
hretcanu2013, ozkanpeltek, ozkanvd},

\item If $a=3$ and $b=1$, $\rho _{3,1}=\dfrac{3+\sqrt{13}}{2}$, we get the
Bronze mean \cite{hretcanu2013, yilmaz},

\item If $a=4$ and $b=1$, $\rho _{4,1}=2+\sqrt{5}$, we get the Subtle mean 
\cite{naschie, hretcanu2013},

\item If $a=1$ and $b=2$, $\rho _{1,2}=2$, we get the Copper mean \cite%
{hretcanu2013},

\item If $a=1$ and $b=3$, $\rho _{1,3}=\dfrac{1+\sqrt{13}}{2}$, we get the
Nickel mean \cite{hretcanu2013}.
\end{itemize}

Throughout this paper, all manifolds are assumed to be of class $C^{\infty }$
unless otherwise stated.

\section{Metallic Structures on Manifolds}

Before making the description of the metallic structures on manifolds, we
give some definitions in order to state the main result of this paper.

\begin{definition}[\protect\cite{hretcanu2013}]
Let $\ M_{n}$ be an $n-$dimensional manifold. A polynomial structure on a
manifold $M_{n}$ is called a metallic structure if it is determined by an $%
\left( 1,1\right) -$tensor field $J$, which satisfies the equation%
\begin{equation}
J^{2}=aJ+bI  \label{3.1}
\end{equation}%
where $a,b$ are positive integers and $I$ is the identity operator on the
Lie algebra $\chi \left( M_{n}\right) $ of the vector fields on $M_{n}$.
\end{definition}

The following proposition gives the main properties of a metallic structure.

\begin{proposition}[\protect\cite{hretcanu2013}]
(i) The eigenvalues of $J$ are the metallic ratio $\rho _{a,b}$ and $a-\rho
_{a,b}$.

(ii) A metallic structure $J$ is an isomorphism on the tangent space of the
manifold $T_{x}\left( M_{n}\right) $ for every $x\in M_{n}$.

(iii) $J$ is invertible and its inverse $J^{-1}=\hat{J}$ satisfies%
\begin{equation*}
b\hat{J}^{2}+a\hat{J}-I=0.
\end{equation*}
\end{proposition}

An almost tangent structure $T$, an almost product structure $F,$ and an
almost complex structure $C$ appear in pairs. In other words, $-T,$ $-F$ and 
$-C$ are also an almost tangent structure, an almost product structure, and
an almost complex structure, respectively \cite{golden}.\ 

Proceeding from this point, it can be said that metallic structures appear
in pairs:

\begin{proposition}
If $J$ is a metallic structure, then $\tilde{J}=aI-J$ is also a metallic
structure.
\end{proposition}

It is known from \cite{petridis} that a polynomial structure on a manifold $%
M_{n}$ induces a generalized almost product structure $F$. Therefore, we
make a connection between the metallic structure and the almost product
structure on $M_{n}$.

\begin{proposition}[\protect\cite{hretcanu2013}]
\label{teo 1}If $F$ is an almost product structure on $M_{n}$, then 
\begin{equation}
\begin{array}{cc}
J= & \dfrac{a}{2}I+\left( \dfrac{2\rho _{a,b}-a}{2}\right) F%
\end{array}
\label{3.2}
\end{equation}%
is a metallic structure on $M_{n}.$

Contrarily, if $J$ is a metallic structure on $M_{n},$ then%
\begin{equation}
\begin{array}{cc}
F= & \left( \dfrac{2}{2\rho _{a,b}-a}J-\dfrac{a}{2\rho _{a,b}-a}I\right)
\bigskip%
\end{array}
\label{3.3}
\end{equation}%
is an almost product structure on $M_{n}.$
\end{proposition}

By the use of the relationship between the almost product structure and the
metallic structure in Proposition \ref{teo 1}, we can give the following
definitions:

\begin{definition}
Let $T$ be an almost tangent structure on $M_{n}.$ Then%
\begin{equation*}
J_{t}=\frac{a}{2}I+\left( \frac{2\rho _{a,b}-a}{2}\right) T
\end{equation*}%
is called a tangent metallic structure on $\left( M_{n},T\right) $.

This structure verifies the following equation:%
\begin{equation*}
J_{t}^{2}-aJ_{t}+\frac{a^{2}}{4}I=0\text{.}
\end{equation*}%
Considering this equation in the real field $%
\mathbb{R}
$, i.e. $x^{2}-ax+\dfrac{a^{2}}{4}=0$, we have the tangent real metallic
ratio $\rho _{t}=\dfrac{a}{2}.$
\end{definition}

\begin{definition}
Let $C$ be an almost complex structure on $M_{n}$. Then 
\begin{equation*}
\begin{array}{cc}
J_{c}= & \dfrac{a}{2}I+\left( \dfrac{2\rho _{a,b}-a}{2}\right) C\bigskip%
\end{array}%
\end{equation*}%
is called a complex metallic structure on $\left( M_{n},C\right) $.

The polynomial equation satisfied by $J_{c}$ is%
\begin{equation*}
J_{c}^{2}-aJ_{c}+\frac{a^{2}+2b}{2}I=0.
\end{equation*}%
For the case of $M_{n}=%
\mathbb{R}
$, we get the equation%
\begin{equation*}
x^{2}-ax+\frac{a^{2}+2b}{2}=0
\end{equation*}%
with solutions,%
\begin{equation*}
\begin{array}{cc}
x_{1}=\dfrac{a}{2}+\dfrac{\sqrt{a^{2}+4b}}{2}i, & x_{2}=\dfrac{a}{2}-\dfrac{%
\sqrt{a^{2}+4b}}{2}i.%
\end{array}%
\end{equation*}%
The positive root of equation 
\begin{equation*}
\rho _{c}=\frac{a}{2}+\frac{\sqrt{a^{2}+4b}}{2}i
\end{equation*}%
is called a complex metallic ratio.
\end{definition}

\section{Examples of The Metallic Structures}

In this section, we give some examples of the metallic structure.

\begin{example}[Clifford algebras]
Let $C^{\prime }(n)$ be the real Clifford algebra of the positive definite
form $\sum_{i=1}^{n}(x^{i})^{2}$ of $%
\mathbb{R}
^{n}$ \cite{lounesto}$.$ According to the Clifford product, the standard
base of $%
\mathbb{R}
^{n}$satisfies the following relations:%
\begin{equation}
\begin{array}{ccc}
e_{i}^{2}=1 & , & i=j \\ 
e_{i}e_{j}+e_{j}e_{i}=0 & , & i\neq j%
\end{array}%
.  \label{4.1}
\end{equation}%
Therefore, using $J_{i}=\dfrac{1}{2}\left( a+\sqrt{a^{2}+4b}e_{i}\right) $
and (\ref{4.1}), we derive a new representation of the Clifford algebra:%
\begin{equation*}
\left\{ 
\begin{array}{lcc}
J_{i},\quad \text{metallic structure} & , & i=j \\ 
J_{i}J_{j}+J_{j}J_{i}=a\left( J_{i}+J_{j}\right) -\dfrac{a^{2}}{2} & , & 
i\neq j.%
\end{array}%
\right.
\end{equation*}%
In \cite{lounesto}, $e_{1}$ and $e_{2}$, orthonormal basis vectors of $%
\mathbb{R}
_{2}^{2},$ are determined by%
\begin{equation*}
\begin{array}{ccccc}
1=I_{2} & , & e_{1}\simeq \left( 
\begin{array}{cc}
1 & 0 \\ 
0 & -1%
\end{array}%
\right) & , & e_{2}\simeq \left( 
\begin{array}{cc}
0 & 1 \\ 
1 & 0%
\end{array}%
\right)%
\end{array}%
\end{equation*}%
and hence we get%
\begin{equation}
\begin{array}{ccccc}
(i) & J_{1} & = & \dfrac{1}{2}\left( a+\sqrt{a^{2}+4b}e_{1}\right) & =\left( 
\begin{array}{cc}
\frac{a+\sqrt{a^{2}+4b}}{2} & 0 \\ 
0 & \frac{a-\sqrt{a^{2}+4b}}{2}%
\end{array}%
\right) \\ 
&  & = & \left( 
\begin{array}{cc}
\rho _{a,b} & 0 \\ 
0 & a-\rho _{a,b}%
\end{array}%
\right) &  \\ 
(ii) & J_{2} & = & \dfrac{1}{2}\left( a+\sqrt{a^{2}+4b}e_{2}\right) & =%
\dfrac{1}{2}\left( 
\begin{array}{cc}
a & \sqrt{a^{2}+4b} \\ 
\sqrt{a^{2}+4b} & a%
\end{array}%
\right) .%
\end{array}
\label{7.7}
\end{equation}
\end{example}

\begin{example}[2D Metallic matrices]
For $J\in 
\mathbb{R}
_{n}^{n}$, if $J$ provides the equation 
\begin{equation}
J^{2}=aJ+bI_{n}  \label{4.14}
\end{equation}%
then this matrix is called a metallic matrix where $I_{n}$ is the identity
matrix on $%
\mathbb{R}
_{n}^{n}.$

In $%
\mathbb{R}
_{2}^{2},$ we obtain a two-parametric family of the metallic structures by
solving (\ref{4.14}),

$i)$ For $r,t$ $\in 
\mathbb{R}
$ and $s\in 
\mathbb{R}
-\left\{ 0\right\} ,$%
\begin{equation}
\begin{array}{ccc}
J_{r,s}= & \left( 
\begin{array}{cc}
r & -\frac{1}{s}\left( r^{2}-ar-b\right) \\ 
s & a-r%
\end{array}%
\right) & or \\ 
J_{s,t}= & \left( 
\begin{array}{cc}
a-t & -\frac{1}{s}\left( t^{2}-at-b\right) \\ 
s & t%
\end{array}%
\right) & .%
\end{array}
\label{4.15}
\end{equation}%
ii) For $r=\rho _{a,b},$ $s\in 
\mathbb{R}
,$ 
\begin{equation*}
\begin{array}{cccc}
J_{\rho _{a,b},s}=\left( 
\begin{array}{cc}
\rho _{a,b} & 0 \\ 
s & a-\rho _{a,b}%
\end{array}%
\right) & or & J_{a-\rho _{a,b},s}=\left( 
\begin{array}{cc}
a-\rho _{a,b} & 0 \\ 
s & \rho _{a,b}%
\end{array}%
\right) & or \\ 
J_{\rho _{a,b},s}=\left( 
\begin{array}{cc}
\rho _{a,b} & s \\ 
0 & a-\rho _{a,b}%
\end{array}%
\right) & or & J_{a-\rho _{a,b},s}=\left( 
\begin{array}{cc}
a-\rho _{a,b} & s \\ 
0 & \rho _{a,b}%
\end{array}%
\right) . & 
\end{array}%
\end{equation*}%
iii) For $r=\rho _{a,b},$ $s=0,$ 
\begin{equation*}
\begin{array}{cccc}
J_{\rho _{a,b},0}=\left( 
\begin{array}{cc}
\rho _{a,b} & 0 \\ 
0 & a-\rho _{a,b}%
\end{array}%
\right) & or & J_{a-\rho _{a,b},s}=\left( 
\begin{array}{cc}
a-\rho _{a,b} & 0 \\ 
0 & \rho _{a,b}%
\end{array}%
\right) . & 
\end{array}%
\end{equation*}%
Then, from (\ref{7.7}) and (\ref{4.15}), we get%
\begin{equation*}
\begin{array}{ccc}
J_{1}=\lim\limits_{s\rightarrow 0}J_{\rho _{a,b},s} & \text{and} & J_{2}=J_{%
\frac{a}{2},\frac{\sqrt{a^{2}+4b}}{2}}.%
\end{array}%
\end{equation*}
\end{example}

\begin{example}[Metallic reflection]
In an Euclidean space $E,$ the reflection with respect to a hyperplane $H$
with the normal $v\in E\backslash \left\{ 0\right\} $ has the formula 
\begin{equation*}
\begin{array}{cc}
r_{v}(x)=x-\dfrac{2\left\langle x,v\right\rangle }{\left\langle
v,v\right\rangle }v & \text{, }x\in E.%
\end{array}%
\end{equation*}%
For $r_{v}^{2}=I_{E}$, $I_{E}$ is the identity on $E$ \cite{processi}.

We can define a metallic reflection with respect to $v$ as%
\begin{equation*}
J_{v}=\frac{1}{2}\left( aI_{E}+\sqrt{a^{2}+4b}r_{v}\right)
\end{equation*}%
where $v$ is an eigenvector of $J_{v}$ with the corresponding eigenvalue $%
a-\rho _{a,b}$. In \cite[p.314]{processi}, under the condition that $X$ is
an orthogonal group of $E,$ we can give the following:

r transformation can be stated explicitly as 
\begin{equation*}
J_{v}(x)=\rho _{a,b}x-\frac{\sqrt{a^{2}+4b}\left\langle x,v\right\rangle }{%
\left\langle v,v\right\rangle }v.
\end{equation*}
\end{example}

\begin{example}[Triple structure in terms of metallic structures]
Let $F$ and $T$ denote two tensor fields of type $\left( 1,1\right) $ on $%
M_{n}$. Using the triple $\left( F,T,K=T\circ F\right) ,$ we obtain the
following structures:\newline
1) $F^{2}=T^{2}=I$ and $T\circ F-F\circ T=0$; then $J^{2}=I$,\newline
2) $F^{2}=T^{2}=I$ and $T\circ F+F\circ T=0$; then $J^{2}=-I$,\newline
3) $F^{2}=T^{2}=-I$ and $T\circ F-F\circ T=0$; then $J^{2}=I$,\newline
4) $F^{2}=T^{2}=-I$ and $T\circ F+F\circ T=0$; then $J^{2}=-I$.\newline
These structures are called almost hyperproduct (ahp), almost biproduct
complex (abpc),almost product bicomplex (apbc), and almost hypercomplex
(ahc), respectively \cite{cruceanu}.\newline
Taking into account (\ref{3.2}), we get%
\begin{equation*}
\begin{array}{lll}
J_{F}=\dfrac{a}{2}I+\left( \dfrac{2\rho _{a,b}-a}{2}\right) F, & J_{T}=%
\dfrac{a}{2}I+\left( \dfrac{2\rho _{a,b}-a}{2}\right) T, & J_{K}=\dfrac{a}{2}%
I+\left( \dfrac{2\rho _{a,b}-a}{2}\right) K.%
\end{array}%
\end{equation*}%
Then, we find a relation between $J_{F}$, $J_{T}$, $J_{K}$ as,%
\begin{equation*}
\sqrt{a^{2}+4b}J_{K}=2J_{T}J_{F}-aJ_{T}-aJ_{F}+\rho _{a,b}^{2}I-bI.
\end{equation*}%
Hence, the triple $\left( J_{F},J_{T},J_{K}\right) $ is,\newline
1$^{\prime }$) An (ahp)-structure if and only if $J_{F},$ $J_{T}$ are
metallic structures\ and $J_{F}J_{T}=J_{T}J_{F},$ then $J_{K}$ is a metallic
structure.$\newline
$2$^{\prime }$) An (abpc)-structure if and only if $J_{F},$ $J_{T}$ are
metallic structures\ and $J_{T}J_{F}+J_{F}J_{T}=a\left( J_{T}+J_{F}\right) -%
\dfrac{1}{2}a^{2}I,$ then $J_{K}$ is a complex metallic structure.\newline
3$^{\prime }$) An (apbc)-structure if and only if $J_{F},$ $J_{T}$ are
complex metallic structures\ and $J_{F}J_{T}=J_{T}J_{F},$ then $J_{K}$ is a
metallic structure.\newline
4$^{\prime }$) An (ahc)-structure if and only if $J_{F},$ $J_{T}$ are
complex metallic structures\ and $J_{T}J_{F}+J_{F}J_{T}=a\left(
J_{T}+J_{F}\right) -\dfrac{1}{2}a^{2}I$ then, $J_{K}$ is a complex metallic
structure.
\end{example}

\begin{example}[Quaternion algebras]
Let $\tilde{H}$ be a quaternion algebra. The base of $\tilde{H}$ is $\left\{
1,e_{1},e_{2},e_{3}\right\} ,$ which verifies 
\begin{equation*}
\begin{array}{lll}
& e_{1}^{2}=e_{2}^{2}=e_{3}^{2}=-1\bigskip &  \\ 
e_{1}e_{2}=-e_{2}e_{1}=e_{3}, & e_{3}e_{1}=-e_{1}e_{3}=e_{2}, & 
e_{2}e_{3}=-e_{3}e_{2}=e_{1}.%
\end{array}%
\end{equation*}
\end{example}

\begin{equation*}
q=a_{0}1+a_{1}e_{1}+a_{2}e_{2}+a_{3}e_{3}
\end{equation*}%
is a quaternion, which can be divided into two parts, such as $S_{q}$ is a
scalar part and $\vec{V}_{q}$ is a vector part. In this way, $%
\begin{array}{l}
q=S_{q}+\vec{V}_{q}%
\end{array}%
$ where $S_{q}=a_{0}$ and $\vec{V}_{q}=a_{1}e_{1}+a_{2}e_{2}+a_{3}e_{3}.$

The norm of a quaternion is determined by $N_{q}=\sqrt{%
a_{0}^{2}+a_{1}^{2}+a_{2}^{2}+a_{3}^{2}}$. Also,\newline
$q_{0}=\frac{q}{N_{q}}$ is a unit quaternion $\left( q\neq 0\right) .$ A
unit form can be written by $q_{0}=\cos \alpha +\vec{S}_{0}\sin \alpha $
where $\vec{S}_{0}$ is a unit vector verifying the equality $\vec{S}%
_{0}^{2}=-1.$

Hence, by the inspire of \cite{yardimci} we can define the metallic
biquaternion structure as

$%
\begin{array}{lll}
J_{q}=\dfrac{a}{2}+\dfrac{\sqrt{a^{2}+4b}i}{2}\vec{S}_{0} & \text{where} & 
\vec{S}_{0}^{2}=-1\text{ and \ }i^{2}=-1.%
\end{array}%
$

Similarly, we can define the metallic split quaternion structure as

$%
\begin{array}{lll}
J_{q}=\dfrac{a}{2}+\vec{S}_{0}\dfrac{\sqrt{a^{2}+4b}}{2} & \text{where} & 
\vec{S}_{0}^{2}=1.%
\end{array}%
$

\section{Connection As Metallic Structure}

\subsection{Connections in principal fibre bundles}

Let $P\left( \pi ,M_{n},G\right) $ be a principal fibre bundle where $\pi
:P\rightarrow M_{n}$ is a fibre projection and $G$ is a Lie group. $V=\ker
\pi _{\ast }$ is the vertical distribution on $P,$ $H$ is the complementary
distribution of $V$, i.e. $TP=V\oplus H$, and $H$ is $G-$invariant. The
corresponding projectors of $V$ and $H$ are $v$ and $h,$ respectively. So,
the tensor field of type $\left( 1,1\right) $%
\begin{equation*}
F=v-h
\end{equation*}%
is an almost product structure on $P$. From \cite{golden}, $F$ defines a
connection if and only if the followings are satisfied:

1) $F\left( X\right) =X\Longleftrightarrow X\in V$

2) $dR_{e}\circ F_{u}=F_{ue}\circ dR_{e}$ for all $e\in G$ and $u\in P$.

By use of the relation between the almost product structure and the metallic
structure, we get the following assertion:

\begin{proposition}
The metallic structure $J$ on $P$ indicates a connection if and only if the
following properties are satisfied:

1$^{\prime }$) $X\in \chi \left( P\right) ,$ where $\chi \left( P\right) $
is the Lie algebra of vector fields on $P,$ is an eigenvector of $J$ with
respect to the eigenvalue $\rho _{a,b}$ if and only if $X$ is a vertical
vector field.

2$^{\prime }$) $dR_{e}\circ J_{u}=J_{ue}\circ dR_{e}$ for all $e\in G$ and $%
u\in P$.
\end{proposition}

Let $\omega \in \Omega ^{1}\left( P,g\right) $ be the connection $1-$form of 
$H$ and suppose that $\Omega \in \Omega ^{2}\left( P,g\right) $ be the
curvature form of $\omega $ where $g$ is the Lie algebra of $G$. Then, we
get \cite{golden},

\begin{equation}
\Omega \left( X,Y\right) =-\frac{1}{4}\omega \left( N_{F}\left( X,Y\right)
\right)  \label{5.3}
\end{equation}%
where $N_{F}$ is the Nijenhuis tensor of $F$.

\begin{proposition}
Using (\ref{3.2}), we obtain%
\begin{equation}
N_{F}=\frac{4}{\left( 2\rho _{a,b}-a\right) ^{2}}N_{J}  \label{5.4}
\end{equation}%
and then we get%
\begin{equation*}
\Omega \left( X,Y\right) =-\frac{1}{\left( 2\rho _{a,b}-a\right) ^{2}}\omega
\left( N_{J}\left( X,Y\right) \right) .
\end{equation*}%
The connection is flat if its curvature form $\Omega \equiv 0.$
\end{proposition}

\begin{proposition}
The connection is flat if and only if the associated metallic structure is
integrable, i.e. $N_{J}\equiv 0$.
\end{proposition}

This connection yields a lift $l_{\omega }:\chi \left( M_{n}\right)
\rightarrow \chi \left( P\right) $ verifying%
\begin{equation}
\left[ l_{\omega }X^{\ast },l_{\omega }Y^{\ast }\right] -l_{\omega }\left[
X^{\ast },Y^{\ast }\right] =N_{F}\left( l_{\omega }X^{\ast },l_{\omega
}Y^{\ast }\right)  \label{5.6*}
\end{equation}%
for every $X^{\ast },$ $Y^{\ast }\in \chi \left( M_{n}\right) $ \cite{golden}%
.

Hence, using (\ref{5.3}) and (\ref{5.6*}):

\begin{proposition}
The lift defined by $l_{\omega }$ is a morphism if and only if the
associated metallic structure is integrable.
\end{proposition}

\subsection{Connection in tangent bundles}

Let $\pi :TM_{n}\rightarrow M_{n}$ be the tangent bundle of $M_{n}$. $%
V\left( M_{n}\right) =\ker \pi _{\ast }$ $\left( \pi _{\ast
}:TTM_{n}\rightarrow TM_{n}\right) $ is called vertical distribution of $%
M_{n}$. $\left( x^{i}\right) _{1\leq i\leq n}$ is a local coordinate system
on $M_{n}$ and $\left( x,y\right) =\left( x^{i},y^{i}\right) _{1\leq i\leq
n} $ is a local coordinate system on $TM_{n}$. For an atlas on $TM_{n}$ with
these local coordinates, the tangent structure of $TM_{n}$ is $T=\frac{%
\partial }{\partial y^{i}}\otimes dx^{i},$ i.e.%
\begin{equation*}
\begin{array}{ccc}
T\left( \frac{\partial }{\partial x^{i}}\right) =\frac{\partial }{\partial
y^{i}} & , & T\left( \frac{\partial }{\partial y^{i}}\right) =0.%
\end{array}%
\end{equation*}%
$v:\chi \left( M_{n}\right) \rightarrow \chi \left( M_{n}\right) $ is a $%
\left( 1,1\right) $ tensor field verifying%
\begin{equation*}
\left\{ 
\begin{array}{c}
T\circ v=0 \\ 
v\circ T=T%
\end{array}%
\right.
\end{equation*}%
which is called the vertical projector.

\begin{definition}[\protect\cite{golden}]
A complementary distribution $N$ to the vertical distribution $V\left(
M_{n}\right) $%
\begin{equation}
\chi \left( M_{n}\right) =N\oplus V\left( M_{n}\right)  \label{5.8}
\end{equation}%
is called normalization, horizontal distribution or non-linear connection.
\end{definition}

Having in mind that a vertical projector $v$ is $C^{\infty }\left(
M_{n}\right) -$linear with $imv=V\left( M_{n}\right) $, we get:

\begin{proposition}[\protect\cite{golden}]
With the help of $\ker v=N\left( v\right) $, a vertical vector $v$ yields a
non-linear connection. Viceversa, if $N$ is a non-linear connection, then $%
h_{N}$ and $v_{N}$ are horizontal and vertical projectors, respectively,
with respect to the decomposition (\ref{5.8}).
\end{proposition}

Then we give the following proposition:

\begin{proposition}[\protect\cite{golden}]
If $N$ is a non-linear connection then $v_{N}$ is a vertical projector with $%
N\left( v_{N}\right) =N$.
\end{proposition}

\begin{definition}[\protect\cite{golden}]
If the following relations hold, a $\left( 1,1\right) $ tensor field $\Gamma 
$ is called non-linear connection of an almost product type:%
\begin{equation*}
\left\{ 
\begin{array}{c}
\Gamma \circ T=-T \\ 
T\circ \Gamma =T%
\end{array}%
\right. .
\end{equation*}%
Let $\Gamma $ be a nonlinear connection of an almost product type, the
following assertions are true:\newline
i) $v_{\Gamma }=\frac{1}{2}\left( I_{\chi \left( M_{n}\right) }-\Gamma
\right) $ is a vertical vector,\medskip \newline
ii) While $N\left( v_{\Gamma }\right) $ is the $\left( +1\right) -$%
eigenspace of $\Gamma ,$ $V\left( M_{n}\right) $ is the $\left( -1\right) -$%
eigenspace of $\Gamma $.
\end{definition}

As a result, every vertical projector $v$ yields a non-linear connection of
an almost product type as $\Gamma =I_{\chi \left( M_{n}\right) }-2v$. Using
the last relation, we have $\Gamma ^{2}=I_{\chi \left( M_{n}\right) }$ $%
(\Gamma $ tensor field is an almost product structure on $M_{n})$.

Hence, we associate this result with the metallic structure.

\begin{proposition}
A non-linear connection $N$ on $M_{n}$, given by the vertical vector $v,$
can also be defined by a metallic structure $J\left( =J_{\Gamma }\right) $%
\begin{equation*}
J=\rho _{a,b}I_{\chi \left( M_{n}\right) }-\sqrt{a^{2}+4b}v
\end{equation*}%
with $N$ the $\rho _{a,b}-$eigenspace and $V\left( M_{n}\right) $ the $%
\left( a-\rho _{a,b}\right) -$eigenspace.
\end{proposition}

\section{Integrability and Parallelism of Metallic Structures}

In this section, we give integrability and parallelism conditions of the
metallic structure. Recall the Nijenhuis tensor of $J,$ 
\begin{equation*}
N_{J}\left( X,Y\right) =J^{2}\left[ X,Y\right] +\left[ JX,JY\right] -J\left[
JX,Y\right] -J\left[ X,JY\right]
\end{equation*}%
for every $X,Y\in \chi \left( M_{n}\right) $.

Let $L$ and $M$ be two complementary distributions on $M_{n}$ corresponding
to $\rho _{a,b}$ and $a-\rho _{a,b},$ respectively. The corresponding
projection operators of $L$ and $M$ are denoted by $l$ and $m$, which
results in

\begin{equation}
\left\{ 
\begin{array}{cc}
l^{2}=l, & m^{2}=m \\ 
lm=ml=0, & l+m=I%
\end{array}%
\right. .  \label{7.4}
\end{equation}%
According to the above conditions and based on a straightforward computation
from (\ref{3.2}), we obtain the following equations:

\begin{equation}
\begin{array}{ccc}
l & = & \dfrac{1}{2\rho _{a,b}-a}J-\dfrac{a-\rho _{a,b}}{2\rho _{a,b}-a}%
I\bigskip \\ 
m & = & -\dfrac{1}{2\rho _{a,b}-a}J+\dfrac{\rho _{a,b}}{2\rho _{a,b}-a}I.%
\end{array}
\label{6.2}
\end{equation}

\begin{theorem}
i) If $N_{J}=0,$ $J$ is integrable.

ii) If $m\left[ lX,lY\right] =0,$ the distribution $L$ is integrable.
Similarly, if $l\left[ mX,mY\right] =0,$ the distribution $M$ is integrable
for all vector fields $X,Y$ on $\chi \left( M_{n}\right) $.
\end{theorem}

\begin{remark}
From (\ref{3.1}) and (\ref{6.2}),%
\begin{equation}
\begin{array}{c}
Jl=lJ=\dfrac{\rho _{a,b}}{2\rho _{a,b}-a}J\bigskip +\dfrac{b}{2\rho _{a,b}-a}%
I=\rho _{a,b}l \\ 
Jm=mJ=-\dfrac{a-\rho _{a,b}}{2\rho _{a,b}-a}J-\dfrac{b}{2\rho _{a,b}-a}%
I=\left( a-\rho _{a,b}\right) m.%
\end{array}
\label{6.6}
\end{equation}%
Using (\ref{6.6}), we can easily see that
\end{remark}

\begin{equation*}
\begin{array}{ccc}
l\left[ mX,mY\right] & = & \dfrac{1}{\left( 2\rho _{a,b}-a\right) ^{2}}%
lN_{J}\left( mX,mY\right) ,\bigskip \\ 
m\left[ lX,lY\right] & = & \dfrac{1}{\left( 2\rho _{a,b}-a\right) ^{2}}%
mN_{J}\left( lX,lY\right) .\bigskip%
\end{array}%
\end{equation*}

\begin{proposition}
Using (\ref{5.4}), a metallic structure $J$ is integrable if and only if the
associated almost product structure $F$ is integrable.
\end{proposition}

\begin{proposition}
Let $X,$ $Y\in \chi \left( M_{n}\right) .$ $L$ is integrable if and only if $%
mN_{J}\left( lX,lY\right) =0$ and $M$ is integrable if and only if $%
lN_{J}\left( mX,mY\right) =0.$ If $J$ is integrable, then the distributions $%
L$ and $M$ are integrable.
\end{proposition}

Now, we give the parallelism conditions of the metallic structures.

Let $\nabla $ be a linear connection on $M_{n}.$ We associate the pair $%
\left( J,\nabla \right) $ with other linear connections, which are Schouten
and Vr\u{a}nceanu connections \cite{bejancu, ianus}.

i) The Schouten connection,%
\begin{equation}
\tilde{\nabla}_{X}Y=l\left( \nabla _{X}lY\right) +m\left( \nabla
_{X}mY\right) .  \label{6.8}
\end{equation}%
ii) The Vr\u{a}nceanu connection,%
\begin{equation}
\hat{\nabla}_{X}Y=l\left( \nabla _{lX}lY\right) +m\left( \nabla
_{mX}mY\right) +l\left[ mX,lY\right] +m\left[ lX,mY\right] .  \label{6.9}
\end{equation}

\begin{proposition}
The projectors $l,$ $m$ are parallels in terms of Schouten and Vr\u{a}nceanu
connections for every linear connection $\nabla $ on $M_{n}$. Moreover, $J$
is parallel in terms of Schouten and Vr\u{a}nceanu connections.
\end{proposition}

\begin{proof}
For every vector fields $X,Y$ on $\chi \left( M_{n}\right) $ and with the
help of (\ref{7.4}),%
\begin{equation*}
\left( \tilde{\nabla}_{X}l\right) Y=\tilde{\nabla}_{X}lY-l\left( \tilde{%
\nabla}_{X}Y\right) =l\left( \nabla _{X}lY\right) -l\left( \nabla
_{X}lY\right) =0,
\end{equation*}%
\begin{equation*}
\begin{array}{ccl}
\left( \hat{\nabla}_{X}l\right) Y & = & \hat{\nabla}_{X}lY-l\left( \hat{%
\nabla}_{X}Y\right) \\ 
& = & l\left( \nabla _{lX}lY\right) +l\left[ mX,lY\right] -l\left( \nabla
_{lX}lY\right) -l\left[ mX,lY\right] \\ 
& = & 0.%
\end{array}%
\end{equation*}%
Therefore, $l$ is parallel relevant to $\tilde{\nabla}$ and $\hat{\nabla}.$

Similarly, the projector $m$ is parallel with respect to the Schouten and Vr%
\u{a}nceanu connections. From (\ref{6.6}), $J$ is parallel in terms of
Schouten and Vr\u{a}nceanu connections.
\end{proof}

\begin{definition}[\protect\cite{das}]
If $\nabla _{X}Y\in L$ where $X\in \chi \left( M_{n}\right) $ and $Y\in L,$
the distribution $L$ is called parallel with respect to linear connection $%
\nabla .$
\end{definition}

\begin{definition}[\protect\cite{das}]
If $(\Delta J)(X,Y)\in L$ where%
\begin{equation}
(\Delta J)(X,Y)=J\nabla _{X}Y-J\nabla _{Y}X-\nabla _{JX}Y+\nabla _{Y}(JX)
\label{7.6}
\end{equation}%
for $X\in L,$ $Y\in \chi \left( M_{n}\right) ,$ the distribution $L$ is
called $\nabla -half$ $parallel$.
\end{definition}

\begin{definition}[\protect\cite{das}]
If $(\Delta J)(X,Y)\in M$ where $X\in L,$ $Y\in \chi \left( M_{n}\right) ,$
the distribution $L$ is called $\nabla -anti$ $half$ $parallel$.
\end{definition}

\begin{proposition}
The distributions $L$ and $M$ are parallel in terms of Schouten and Vr\u{a}%
nceanu connections for the linear connection $\nabla $ on $M_{n}.$
\end{proposition}

\begin{proof}
Let $X\in \chi \left( M_{n}\right) $ and $Y\in L.$ $mY=0$ and $lY=Y,$ using (%
\ref{6.8}) and (\ref{6.9}) we obtain 
\begin{equation*}
\tilde{\nabla}_{X}Y=l\left( \nabla _{X}Y\right) \in L,
\end{equation*}%
\begin{equation*}
\hat{\nabla}_{X}Y=l\left( \nabla _{lX}Y\right) +l\left[ mX,Y\right] \in L.
\end{equation*}%
Thus, $L$ is parallel with respect to $\tilde{\nabla}$ and $\hat{\nabla}.$

In the same manner, we can see that $M$ satisfies the similar relations.
\end{proof}

\begin{proposition}
The distributions $L$ and $M$ are parallel with respect to $\nabla $ linear
connection if and only if $\nabla $ and $\tilde{\nabla}$ are equal.
\end{proposition}

\begin{proof}
If $L,$ $M$ are $\nabla -parallel$ then $\nabla _{X}(lY)\in L$ and $\nabla
_{X}(mY)\in M$ where $X,Y\in \chi (M_{n}).$ For that reason 
\begin{equation*}
\begin{array}{ccc}
\nabla _{X}(lY)=l\nabla _{X}(lY) & \text{and} & \nabla _{X}(mY)=m\nabla
_{X}(mY).%
\end{array}%
\end{equation*}%
Since, $l+m=I$ and from (\ref{6.8}),%
\begin{equation*}
\nabla _{X}Y=l\nabla _{X}(lY)+m\nabla _{X}(mY)=\tilde{\nabla}_{X}Y.
\end{equation*}%
Therefore $\nabla =\tilde{\nabla}.$

The converse can be shown easily.
\end{proof}

\begin{proposition}
If $\left[ lX,mY\right] \in L$ where $X\in L,$ $Y\in \chi \left(
M_{n}\right) $, the distribution $L$ is half parallel with respect to the Vr%
\u{a}nceanu connection.
\end{proposition}

\begin{proof}
In the view of the equation (\ref{7.6}) for $\hat{\nabla},$ we get%
\begin{equation*}
m(\Delta J)(X,Y)=mJ\hat{\nabla}_{X}Y-mJ\hat{\nabla}_{Y}X-m\hat{\nabla}%
_{JX}Y+m\hat{\nabla}_{Y}(JX)
\end{equation*}%
where $X\in L,$ $Y\in \chi \left( M_{n}\right) .$

As a result, by (\ref{6.6}) and (\ref{6.9}), we have 
\begin{equation*}
m(\Delta J)(X,Y)=\left( a-2\rho _{a,b}\right) m\left[ lX,mY\right]
\end{equation*}%
which proves the proposition.
\end{proof}

Similarly, we have the following proposition for distribution $M:$

\begin{proposition}
If $\left[ mX,lY\right] \in M$ where $X\in M,$ $Y\in \chi \left(
M_{n}\right) $, the distribution $M$ is half parallel with respect to the Vr%
\u{a}nceanu connection.
\end{proposition}

\begin{proposition}
The distributions $L$ and $M$ are anti half parallel with respect to Vr\u{a}%
nceanu connection.
\end{proposition}

\begin{proof}
Taking account of the equation (\ref{7.6}) for $\hat{\nabla},$ we get%
\begin{equation*}
l(\Delta J)(X,Y)=lJ\hat{\nabla}_{X}Y-lJ\hat{\nabla}_{Y}X-l\hat{\nabla}%
_{JX}Y+l\hat{\nabla}_{Y}(JX)
\end{equation*}%
where $X\in L,$ $Y\in \chi \left( M_{n}\right) .$

Using (\ref{6.6}) and (\ref{6.9}), we get%
\begin{equation*}
l(\Delta J)(X,Y)=\left( 2\rho _{a,b}-a\right) l\left[ mX,lY\right] .
\end{equation*}%
Because $mX=0,$ we get $l(\Delta J)(X,Y)=0.$ Thus, $(\Delta J)(X,Y)\in M.$

In the same way, we can show that $M$ is anti-half parallel with respect to
the Vr\u{a}nceanu connection.
\end{proof}

\section{Metallic Riemannian Manifolds}

In this section, we study a metallic Riemannian manifold with respect to the
Riemannian metric and we give an example of the metallic structure on
manifold $%
\mathbb{R}
^{2}$.

\begin{definition}[\protect\cite{gray, pri, yano}]
Let $F$ be an almost product structure on $M_{n}$ and $g$ be a Riemannian
metric (or a semi-Riemannian metric) given by 
\begin{equation*}
\begin{array}{ccc}
g\left( F\left( X\right) ,F\left( Y\right) \right) =g\left( X,Y\right)  & ,
& \forall X,Y\in \chi \left( M_{n}\right) 
\end{array}%
\end{equation*}%
or equivalently, $F$ be a $g$-symmetric endomorphism, such as 
\begin{equation*}
g\left( F\left( X\right) ,F\left( Y\right) \right) =g\left( X,Y\right) .
\end{equation*}%
So, the pair $\left( g,F\right) $ is called a Riemannian almost product
structure (or a semi-Riemannian almost product structure).
\end{definition}

By the help of (\ref{3.2}) and (\ref{3.3}), the following proposition can be
given:

\begin{proposition}
The almost product structure $F$ is a $g-$symmetric endomorphism if and only
if the metallic structure $J$ is also a $g-$symmetric endomorphism.
\end{proposition}

\begin{definition}[\protect\cite{hretcanu2013}]
Let $g$ be a Riemannian metric $($or a semi-Riemannian metric) on $M_{n},$
such as%
\begin{equation*}
g\left( J\left( X\right) ,Y\right) =g\left( X,J\left( Y\right) \right) ,%
\text{ \ \ \ }\forall \text{ }X,Y\in \chi (M_{n}).
\end{equation*}%
Then, $\left( g,J\right) $ is called a metallic Riemannian structure (or a
metallic semi-Riemannian structure). Also, $\left( M_{n},g,J\right) $ is
named a metallic Riemannian manifold (or a metallic semi-Riemannian
manifold).
\end{definition}

\begin{corollary}
On a metallic Riemannian manifold,

a) The projectors $l$, $m$ are $g-$symmetric, i.e.%
\begin{equation*}
\begin{array}{c}
g\left( l(X),Y\right) =g\left( X,l\left( Y\right) \right) \\ 
g\left( m(X),Y\right) =g\left( X,m\left( Y\right) \right) .%
\end{array}%
\end{equation*}%
b) The distributions $L$, $M$ are $g-$orthogonal, i.e.%
\begin{equation*}
g\left( l(X),m(Y)\right) =0.
\end{equation*}%
c) The metallic structure is $N_{J}-$symmetric, i.e.%
\begin{equation*}
N_{J}\left( J\left( X\right) ,Y\right) =N_{J}\left( X,J\left( Y\right)
\right) .
\end{equation*}%
A Riemannian almost product structure is a locally product structure if $F$
is parallel with respect to the Levi-Civita connection $\overset{g}{\nabla }$%
of $g,$ i.e. $\overset{g}{\nabla }F=0$ and if $\nabla $ is a symmetric
linear connection, then the Nijenhuis tensor of $F$ satisfies%
\begin{equation*}
N_{F}\left( X,Y\right) =\left( \nabla _{FX}F\right) Y-\left( \nabla
_{FY}F\right) X-F\left( \nabla _{X}F\right) Y+F\left( \nabla _{Y}F\right) X.
\end{equation*}%
If $(M_{n},g,J)$ is a locally product metallic Riemannian manifold, then the
metallic structure $J$ is integrable.
\end{corollary}

\begin{theorem}
The set of linear connections $\nabla $ for which $\nabla J=0$ is 
\begin{equation*}
\begin{array}{ccl}
\nabla _{X}Y & = & \dfrac{1}{a^{2}+4b}\left[ \left( a^{2}+2b\right) \bar{%
\nabla}_{X}Y+2J\left( \bar{\nabla}_{X}JY\right) -aJ\left( \bar{\nabla}%
_{X}Y\right) -a\left( \bar{\nabla}_{X}JY\right) \right] \\ 
&  & +O_{F}Q\left( X,Y\right)%
\end{array}%
\end{equation*}%
where $\bar{\nabla}$ is an arbitrary fixed linear connection and $Q$ is a $%
\left( 1,2\right) -$tensor field for which $O_{F}Q$ is an associated Obata
operator 
\begin{equation*}
O_{F}Q\left( X,Y\right) =\frac{1}{2}\left[ Q\left( X,Y\right) +FQ\left(
X,FY\right) \right]
\end{equation*}%
for the corresponding almost product structure (\ref{3.2}).
\end{theorem}

We conclude with the following example for the metallic structure.

\begin{example}
\begin{eqnarray*}
l &=&\frac{x^{2}}{x^{2}+y^{2}}\frac{\partial }{\partial x}\otimes dx+\frac{xy%
}{x^{2}+y^{2}}\frac{\partial }{\partial x}\otimes dy+\frac{xy}{x^{2}+y^{2}}%
\frac{\partial }{\partial y}\otimes dx+\frac{y^{2}}{x^{2}+y^{2}}\frac{%
\partial }{\partial y}\otimes dy \\
m &=&\frac{y^{2}}{x^{2}+y^{2}}\frac{\partial }{\partial x}\otimes dx-\frac{xy%
}{x^{2}+y^{2}}\frac{\partial }{\partial x}\otimes dy-\frac{xy}{x^{2}+y^{2}}%
\frac{\partial }{\partial y}\otimes dx+\frac{x^{2}}{x^{2}+y^{2}}\frac{%
\partial }{\partial y}\otimes dy
\end{eqnarray*}%
are projection operators in $%
\mathbb{R}
^{2}$ which satisfy the conditions (\ref{7.4}).%
\begin{equation*}
\begin{array}{ccc}
L=Sp\left\{ x\frac{\partial }{\partial x}+y\frac{\partial }{\partial y}%
\right\} & \text{and} & M=Sp\left\{ y\frac{\partial }{\partial x}-x\frac{%
\partial }{\partial y}\right\}%
\end{array}%
\end{equation*}%
are complementary distributions corresponding to the projection operators $l$
and $m,$ respectively. The distributions $L$, $M$ are orthogonal with
respect to the Euclidean metric of $%
\mathbb{R}
^{2}.$ Furthermore, these distributions are connected to the metallic
structure 
\begin{eqnarray*}
J\left( \frac{\partial }{\partial x}\right) &=&\frac{\rho
_{a,b}x^{2}+(a-\rho _{a,b})y^{2}}{x^{2}+y^{2}}\frac{\partial }{\partial x}+%
\frac{\left( 2\rho _{a,b}-a\right) xy}{x^{2}+y^{2}}\frac{\partial }{\partial
y}, \\
J\left( \frac{\partial }{\partial y}\right) &=&\frac{\left( 2\rho
_{a,b}-a\right) xy}{x^{2}+y^{2}}\frac{\partial }{\partial x}+\frac{(a-\rho
_{a,b})x^{2}+\rho _{a,b}y^{2}}{x^{2}+y^{2}}\frac{\partial }{\partial y},
\end{eqnarray*}

which is integrable because $N_{J}\left( \frac{\partial }{\partial x},\frac{%
\partial }{\partial y}\right) =0.$
\end{example}

\end{document}